\documentclass{amsart}

\usepackage[utf8]{inputenc}
\usepackage[margin=2cm,includeheadfoot,centering]{geometry} 
\usepackage{array}
\usepackage{xcolor}
\usepackage{cite}
\usepackage{tikz}
\usepackage{soul}
\usepackage{placeins}
\usepackage{fancyvrb} 

\usetikzlibrary{mindmap, backgrounds}
\usepackage{xcolor}
\usetikzlibrary{matrix,arrows,decorations.pathmorphing,automata, matrix, positioning, calc, shapes.multipart,decorations.pathreplacing,calligraphy}
\usepackage{graphicx}
\usepackage[config, labelfont={normalsize}]{caption,subfig}
\usepackage{mathrsfs}
\definecolor{green}{RGB}{0,127,0}
\definecolor{red}{RGB}{191,0,0}
\usepackage[colorlinks,cite color=red,link color=green,pagebackref=true]{hyperref}
\usepackage[shortlabels]{enumitem}
\usepackage[vcentermath]{youngtab}
\usepackage{epigraph}
\usepackage{ytableau}
\usepackage[capitalize]{cleveref}
\usepackage{mathtools}
\usepackage{algorithm,algorithmicx,algpseudocode}
\usepackage[leqno]{amsmath}
\usepackage{latexsym,amsthm,amsfonts,amssymb,hyperref}
\usepackage[all]{xy} \SelectTips{eu}{}

\newcommand{\numberseries}{\mdseries}   
\newlength{\thmtopspace}                
\newlength{\thmbotspace}                
\newlength{\thmheadspace}               
\newlength{\thmindent}                  
\newtheoremstyle{bfupright head,upright body}
                {\thmtopspace}{\thmbotspace}
                {\upshape}{\thmindent}{\bfseries}{.}{\thmheadspace}
                {{\numberseries \thmnumber{(#2) }}\thmnote{#3}}

\newtheoremstyle{fixed bf head,slanted body}
                {\thmtopspace}{\thmbotspace}{\slshape}
                {\thmindent}{\bfseries}{.}{\thmheadspace}
                {{\numberseries \thmnumber{(#2) }}\thmname{#1}\thmnote{ (#3)}}

\newtheoremstyle{fixed bf head,upright body}
                {\thmtopspace}{\thmbotspace}{\upshape}
                {\thmindent}{\bfseries}{.}{\thmheadspace}
                {{\numberseries \thmnumber{(#2) }}\thmname{#1}\thmnote{ (#3)}}

\newtheoremstyle{numbered paragraph}
                {\thmtopspace}{\thmbotspace}{\upshape}
                {\thmindent}{\upshape}{}{0pt}
                {{\numberseries \thmnumber{(#2) }}}

\theoremstyle{bfupright head,upright body}
\newtheorem{res}{}[section]             \newtheorem*{res*}{}
               \newtheorem*{bfhpg*}{}

\theoremstyle{theorem}
\newtheorem{thm}[res]{Theorem}          \newtheorem*{thm*}{Theorem}
\newtheorem{prp}[res]{Proposition}      \newtheorem*{prp*}{Proposition}
\newtheorem{rmk}[res]{Remark}           \newtheorem*{rmk*}{Remark}
            \newtheorem*{lem*}{Lemma}
          \newtheorem*{conj*}{Conjecture}
\newtheorem{example}[res]{Example}          \newtheorem*{example*}{Example}

\theoremstyle{fixed bf head,upright body}
       \newtheorem*{dfn*}{Definition}

           \newtheorem*{fact*}{Fact}

\theoremstyle{numbered paragraph}

\newlength{\thmlistleft}        
\newlength{\thmlistright}       
\newlength{\thmlistpartopsep}   
\newlength{\thmlisttopsep}      
\newlength{\thmlistparsep}      
\newlength{\thmlistitemsep}     

\newcounter{prt}
  {\end{list}}%

\newenvironment{prf*}[1][Proof]{%
  \begin{proof}[\bf #1]
    \setcounter{equation}{0}
    }
  {\end{proof}
}
\renewcommand{\eqref}[1]{\pgref{eq:#1}}
\newcommand{\pgref}[1]{(\ref{#1})}

\newcommand{\op}{\operatorname}

\def\urltilda{\kern -.15em\lower .7ex\hbox{\~{}}\kern .04em}

\def\FF{{\mathbb F}}

\numberwithin{equation}{res}

\newcommand{\<}{\langle}
\newcolumntype{L}{>$l<$}
\newcommand{\rr}{\rangle}

\begin{document}

\title{Residual Intersections and Schubert Varieties}

\author[S. A. Filippini]{Sara Angela Filippini}

\address{University of Salento, Italy}

\email{saraangela.filippini@unisalento.it}

\author[X. Ni]{Xianglong Ni}

\address{University of California, Berkeley}

\email{xlni@berkeley.edu}

\author[J. Torres]{Jacinta Torres}

\address{Uniwersytet Jagiello\'nski, Krak\'ow, Poland}

\email{jacinta.torres@uj.edu.pl}

\author[J. Weyman]{Jerzy Weyman}

\address{Uniwersytet Jagiello\'nski, Krak\'ow, Poland}

\email{jerzy.weyman@uj.edu.pl}


\date{\today}



\begin{abstract}
Inspired by the work of Ulrich \cite{ulrich} and Huneke--Ulrich \cite{hunekeulrich1988}, we describe a pattern to show that the ideals of certain opposite embedded Schubert varieties (defined by this pattern) arise by taking residual intersections of two (geometrically linked) opposite Schubert varieties (which we called {\it Ulrich pairs} in \cite{FTW20}). This pattern is uniform for the ADE types. Some of the free resolutions of the Schubert varieties in question are important for the structure of finite free resolutions. Our proof is representation theoretical and uniform for our pattern, however it is possible to derive our results using case-by-case analysis and the aid of a computer.
\end{abstract}

\maketitle

\thispagestyle{empty}

\section{Introduction}

The theory of linkage and residual intersections has arisen in a natural way in enumerative geometry, intersection theory, the study of Rees rings and integral closures. Linkage has been extensively studied (see \cite{watanabe, rao, ulrich, HU, ulrich2023remarks } and the references therein); in particular it has given rise to the study of licci ideals, that is, the ideals which are in the linkage class of a complete intersection. On the other hand, residual intersections \cite{artin1972residual, cumming2007residual,hunekeulrich1988,kustin1992generating, HU, fulton1984excess}, have remained more obscure, although they have been more extensively studied in recent years \cite{ chardin2001hilbert,corso2002core,bruns1990resolution, boucca2019residual, hassanzadeh2016residual, hassanzadeh2024deformation}. In this paper we construct a family of examples of residual intersections using the defining ideals of certain opposite Schubert varieties defined by a pattern that we describe below. The opposite Schubert varieties in question have been subject of recent investigations \cite{samweyman2021schubert, FTW20}.

Let $D$ be a simply-laced Dynkin diagram and let $k$ be an extremal or minuscule vertex. We consider a rooted graph $\mathcal{G}_{k}$ defined by attaching an extra distinguished vertex to $k$. This graph has a $``Y"$ shape where we display the distinguished vertex at the top, and the edges of $G_k$ correspond to the vertices of $D$. See Figure \ref{graphgeny}. Let $G$ be a simply connected, reductive, complex algebraic group with Lie algebra associated to $D$.  Consider $P_{k} \subset G$ the parabolic subgroup of $G$ stabilising the fundamental weight $\omega_{k}$, and $V(\omega_{k})$ the irreducible finite-dimensional $G$-module of highest weight $\omega_{k}$, or the $k$-th fundamental representation of $G$. To each vertex $p$ in our graph $\mathcal{G}_{k}$ is  associated a homogeneous coordinate in $\mathbb{P}(V(\omega_{k}))$, which is an extremal Pl\"ucker coordinate. This coordinate defines uniquely the opposite Schubert variety $X^{p} $ in the partial flag variety $G/P_{k}$.

Let $u$ be the degree three vertex of $D$, assuming that we are not in type $A_{n}$. Now, if we start a ``walk'' along our Dynkin diagram towards $u$ from our extremal starting vertex $k = x_{c-2}$, once we reach $u$, there are precisely two vertices adjacent to $u$ which we have not visited during our walk. We denote these vertices $y_1$ and $z_1$. If $D$ is a type $A_{n}$ Dynkin diagram, then the vertices $y_1$ and $z_1$ are the vertices adjacent to $k$. These vertices correspond to edges in $\mathcal{G}_{k}$ which we label by the same letters. We denote the two distinct vertices adjacent to each of them by $p_{y_1}$ and $p_{z_1}$, respectively. 

\subsection{Main result}
The starting point and motivation of the present work is the following theorem, which readily follows from the work of Ulrich \cite{ulrich}.

\begin{thm}
\label{thm:ulrichpairs}
The defining ideals of the opposite Schubert varieties corresponding to $p_{y_1}$ and $p_{z_1}$ are linked by the regular sequence formed by the coordinates corresponding to vertices in $\mathcal{G}_{k}$ that lie along the branch which does not contain either $p_{y_1}$ or $p_{z_1}$.
\end{thm}

In \cite{FTW20} we refer to such pairs of opposite Schubert varieties as \textit{Ulrich pairs}. Before we state our main theorem, we need to introduce some notation. \\

We will label the vertices of $\mathcal{G}_{k}$ as defined by the picture in Figure \ref{tpqr}, is known as the $T_{c-1,d+1,t+1}$ diagram, where $c,d,t$ are integers such that $\frac{1}{c-1}+ \frac{1}{d+1} + \frac{1}{t+1} \geq 1$. We will label the nodes of $\mathcal{G}_{k}$ by coordinates $p_{j}$ with $j$ as indicated in Figure \ref{graphgeny}. Below, we denote the opposite Schubert variety corresponding to the Pl\"ucker coordinate $p_j$ by $X^j$.

\begin{figure}
\includegraphics[scale = 0.7]{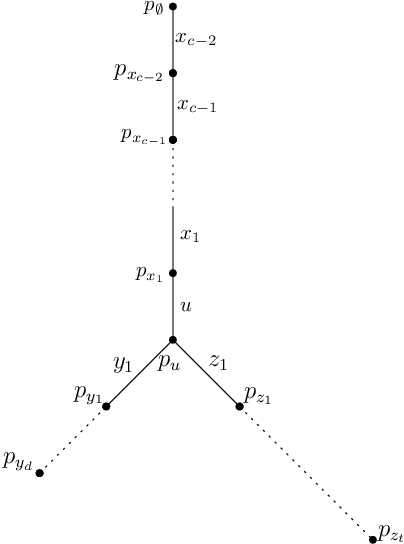}
\caption{The graph $\mathcal{G}_{k}$. }
\label{graphgeny}
\end{figure}

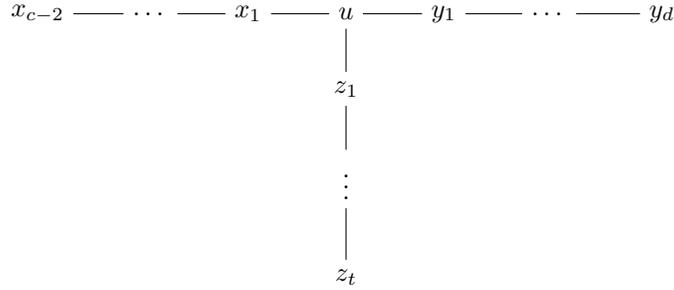
\begin{figure}
\begin{center}
\begin{tikzpicture}
\node (1) at (0,0) {$x_{c-2}$} ;
\node (2) at (1.5,0) {$\dots$} ;
\node (3) at (2.8,0) {$x_{1}$} ;
\node (4) at (4.1,0) {$u$} ;
\node (5) at (5.4,0) {$y_1$} ;
\node (6) at (6.8,0) {$\dots$} ;
\node (7) at (8.3,0) {$y_d$} ;

\node (8) at (4.1,-1) {$z_1$} ;
\node (9) at (4.1,-2.2) {$\vdots$} ;
\node (10) at (4.1,-3.5) {$z_t$} ;

\draw (1)-- (2) -- (3)--(4)--(5)--(6)--(7);
\draw (4) -- (8) -- (9) -- (10);
\end{tikzpicture}
\end{center}
\caption{The graph $T_{c-1,d+1,t+1}$. }
\label{tpqr}
\end{figure}

\begin{thm}\hfill
\label{ourthm}
\begin{enumerate}
\item The defining ideal $I(X^{y_1})$, respectively $I(X^{z_1})$ of the  opposite Schubert variety $X^{y_{1}}$, respectively $X^{z_1}$, is defined, scheme-theoretically,  by the extremal Pl\"ucker coordinates on the right, respectively left arm of $\mathcal{G}_{k}$. 
\item Let $d\ge l\ge 1$. The defining ideal $I(X^{y_l})$ of the  opposite Schubert variety $X^{y_l}$ of codimension $l+c$ on the left arm of the graph $\mathcal{G}_{k}$ is a residual intersection
$$I(X^{y_l})= (p_{\emptyset},\ldots ,p_{y_{l-1}}): I(X^{z_1}).$$
\item Let $t\ge m\ge 1$. The defining ideal $I(X^{z_m})$ of the  opposite Schubert variety $X^{z_m}$ of codimension $m+c$ on the left arm of the graph $\mathcal{G}_{k}$ is a residual intersection
$$I(X^{z_m})= (p_{\emptyset},\ldots ,p_{z_{m-1}}): I(X^{y_1}).$$
\end{enumerate}

\end{thm}

\subsection{Organization of the paper} In Section 2 we collect basics on commutative algebra, and in Section 3 we collect basics on Schubert varieties. In Section 4 we prove our main result, Theorem \ref{ourthm}. Later, in Section 5 we collect non-minuscule examples, exhibiting a method of how one may obtain a case by case proof of our main result which avoids certain general representation theoretic arguments. In Section 6 we present explicit calculations in the minuscule cases, most of them assisted by Macaulay2.  

\section*{Acknowledgements}
J.T. and J.W. were supported by the grant MAESTRO NCN-UMO-2019/34/A/ST1/00263 - Research in Commutative Algebra and Representation
Theory. S.A.F. and J.W. were supported by NAWA POWROTY - \break PPN/PPO/2018/1/00013/U/00001 - Applications of Lie algebras to Commutative Algebra. J.W. was supported by the OPUS grant National Science Centre, Poland grant UMO-2018/29/BST1/01290. J.T. was supported by the grant SONATA NCN UMO-2021/43/D/ST1/02290.

\section{Commutative Algebra basics}\label{sec:commutatuivealgebrabasics}

Let $(R,\frak{m})$ denote a Noetherian local ring (resp. a polynomial ring over a field).
For a proper (resp. homogeneous) ideal $I$ we denote by $\mathrm{ht}(I)$ its {\it height}, by $\mathrm{grade}(I)$ the maximal length of an $R$-regular sequence contained in $I$,  by $\mu(I)$ the minimal number of generators of $I$, and by $d(I) = \mu(I) - \mathrm{grade}(I)$ its deviation. 
An ideal $I$ is said to be a {\it complete intersection} (resp. an {\it almost complete intersection}) if $d(I) = 0$ (resp. if $d(I) = 1$).
 
Furthermore, if $R$ is Cohen-Macaulay with $\dim R =n$, let  $t(I) = \dim \mathrm{Ext}_R^n (R/ {\frak m}, R)$ be the {\it Cohen-Macaulay type} of $I$. An ideal $I$ is {\it perfect} if its grade coincides with its projective dimension $\mathrm{grade}(I) = \mathrm{pd}_R(R/I)$. We say that $I$ is {\it Gorenstein} if $I$ is perfect and $t(I)=1$. The \emph{height} $\mathrm{ht}(I)$ of $I$ is the minimum of the heights of the prime ideals in $R$ that contain it. The height of a \textit{prime} ideal is simply the codimension of the variety defined by it. A proper ideal $I$ in $R$ is called \textit{unmixed} if  $\mathrm{ht}(P) = \mathrm{ht}(I)$ for every associated prime $P$ of $R/I$. As $R$ is Cohen-Macaulay, we have that $\mathrm{grade}(I) = \mathrm{ht}(I)$ for any proper ideal $I$, and by slight abuse of terminology we will often call this the codimension of $I$.\\

\subsection{Linkage} Let $I, J \subset R$ be ideals. The \emph{colon ideal} or \emph{ideal quotient} is $I:J := \left\{r \in R | rJ \subset I \right\}$. The ideals $I,J$ in R are said to be {\it linked} if they are both proper and if there exists a regular sequence $a= \{a_1, \ldots, a_g\}$ in $R$ such that $J = (a):I$ and $I = (a):J$. If $R$ is Cohen--Macaulay and $I, J$ are linked, then $I$ and $J$ are unmixed of the same height $g$. A link is called \textit{geometric} if $\mathrm{ht}(I+J) \geq g+1$. In this case, 

\begin{align}
\label{geometriclinkage}
(a) = I \cap J.
\end{align}

\noindent Moreover, if (\ref{geometriclinkage}) holds, $I$ and $J$ are both unmixed of the same height and have no common prime components, then they are (geometrically) linked. An ideal $I$ is said to be \textit{licci} if it is in the linkage class of a complete intersection. If $\mathrm{grade}(I)=2$, then $I$ is perfect if and only if it is licci \cite{apery, gaeta, PS}. Only the ``if'' implication holds when $\mathrm{grade}(I)\geq 3$. Various other classes of licci ideals have been studied in recent decades \cite{watanabe, guerrieri2022higher, ulrich, HU }.

 
 \subsection{Residual Intersections} Let $R$ above be additionally Cohen--Macaulay, $I \subset R$ an ideal, $g = \mathrm{ht}(I)$ and $s \geq g$. We say that a proper ideal $K$ of $R$ is an $s$-\textit{residual intersection} of $I$ if $K = A:I$ for some ideal $A \subset I$ such that $\mathrm{ht}(K) \geq s \geq \mu(A)$.  Notice that this condition implies that $\mathrm{ht}(A) = g$. If $R$ is Gorenstein and $I$ is unmixed, then the concept of $g$-residual intersection is the same as linkage. \\

  In \cite{H83} Huneke proved that the ideal $I(X)$, generated by the maximal order minors of a generic $r\times s$ matrix $X$, is the residual intersection of a height $2$ perfect ideal.

\begin{example}
\label{grade2}
    For $r\leq s$ let $A$ be an $s\times(s+1)$-matrix and $B$ the submatrix given by its right-most $r$ columns.  Then $I(X) = (a_1, \ldots, a_{s-r+1}:I_s(A))$ is a residual intersection of $I_s(A)$. Note that the codimension of the ideal $I_s(A)$ is $(s-(s-1))\cdot(s +1 - (s-1)) = 2$. Denote by $a_i$ the $s\times s$ minor of $A$ obtained by deleting the $i$-th  column.
\end{example}

 Kustin and Ulrich in \cite[Theorem 10.2]{KU} described residual intersections of grade $3$ Gorenstein ideals.

 \begin{example} 
\label{grade3gorenstein}
 Let $X$ be a generic $(2r + 1)\times(2r + 1)$ skew-symmetric matrix, $Y$ a generic $s\times (2r + 1)$ matrix with $s\geq 3$. Denote by $p_1, \ldots, p_{2r+1}$ the $2r\times 2r$ Pfaffians of $X$. Now consider the skew-symmetric matrix 
\[A = \begin{pmatrix}
X & Y\\
- Y^t & 0
\end{pmatrix},\]
and let $J$ be the ideal generated by the Pfaffians of all sizes of $A$ which contain $X$. Then $J$ is a residual intersection of the codimension $3$ ideal generated by the Pfaffians of $X$.
\end{example}

In Theorem \ref{ourthm} we extend Kustin--Ulrich's result to ADE types, where we naturally go beyond codimension $3$. The following remark follows from the definitions. It will be crucial in the proof of our main result Theorem \ref{ourthm}.

\begin{rmk}
\label{geometricri}
Let $I$ and $K$ be prime ideals of heights $g,s$ respectively such that $g \leq s$ and $I \nsubseteq K$, and $A = I \cap K$. It follows that $K = A:I$, and if $\mu(A) = s$ then by definition $K$ is an $s$-residual intersection.
\end{rmk}

We refer the reader to \cite[Introduction, Sections 4, 10]{KU}, from where we have borrowed much of the material for this section. \\

\section{Basics on Schubert Varieties}\label{sec:schubertbasics}

 Let $G\supset B \supset T$ be a reductive complex algebraic group with simply laced Dynkin diagram, Borel subgroup $B$ and maximal torus $T$, and let $n$ be the rank of $G$. Let $B^{-}$ be the opposite Borel subgroup of $B$, that is, the unique Borel subgroup such that $B \cap B^{-} = T$. Let $X:= \operatorname{Hom}(T,\mathbb{C}^{\times})$ be the integral weight lattice for $T$, and $\left\{\omega_{k}\right\}_{1 \leq k \leq n} \subset X$ be the set of fundamental weights that span $X$ as a lattice. Moreover, let $R\subset X$ be the set of roots, let $W$ be the Weyl group of symmetries of $R$ and $R^{+}\subset R$ be the choice of positive roots corresponding to $B$. We denote by $P_{k} \supset B$, respectively $W_{k} \subset W$ be the corresponding parabolic subgroup of $G$ respectively of $W$ stabilizing $\omega_{k}$. We consider the irreducible, finite dimensional  complex representation of $G$ of highest weight $\omega_{k}$ and denote it by $V(\omega_{k})$.  Now the quotient $G/P_{k}$ has the structure of a smooth projective variety. Moreover, it comes together with a canonical embedding $G/P_{k} \hookrightarrow \mathbb{P}(V(\omega_{k}))$ given by $gP_k \mapsto g v_{\omega_{k}}$, where $v_{\omega_{k}}$ is the highest weight vector. We will call the coordinates in $V(\omega_k)^*$ \textit{Pl\"ucker coordinates}, (borrowing the name from the coordinates of $\mathbb{P}( \bigwedge^{k} \mathbb{C}^{n})$, which with our notation will be the Pl\"ucker coordinates in type $A_{n}$). The \textit{extremal Pl\"ucker coordinates} in $V(\omega_k)^*$ are those dual to the $W$-orbit of $v_{\omega_k}$; we write $p_w$ for the coordinate dual to $wv_{\omega_k}$. To define opposite Schubert varieties in $G/P_{k}$, we consider the cosets $W/W_{k}$ and the Bruhat decomposition \[G/P_{k} = \underset{w \in W/W_{k}}{\bigsqcup} B^{-}wP_{k}. \] 
 The opposite Schubert cells are the $B^{-}$-orbits $C^{w}: = B^{-}wP_{k}$, and their closures $X^{w} : = \overline{B^{-}wP_{k}}$ are the opposite Schubert varieties. If $w$ is a minimal length representative of its coset in $W/W_{k}$, then its length $l(w)$ is the codimension of the opposite Schubert variety $X^{w}$. The cell  $C^{id}$ is called the opposite big open cell. We denote by  $Y^{w}:= X^{w} \cap C^{id}$ the corresponding intersection. 
 
\begin{example}
    \label{ex:gln}
The first example is $G = \operatorname{GL}(n ,\mathbb{C}), T \subset G$ are the invertible diagonal matrices, and $B$ are the upper triangular matrices.  In this case, we have $X \cong \mathbb{Z}^{n}$ with coordinate vectors $\varepsilon_{k}$ and $\omega_{k} = \sum^{k}_{j=1} \varepsilon_{j}$ for $1 \leq k \leq n$; the Weyl group $W = S_{n}$ is the symmetric group on $n$ letters. The fundamental representations are given by $V(\omega_{k})  = \bigwedge^{k} \mathbb{C}^{n}$.  The varieties $G/P_{k}$ are precisely the grassmannians $Gr(k, n)$ equipped with the Pl\"ucker embedding $Gr(k, n) \hookrightarrow \mathbb{P}( \bigwedge^{k} \mathbb{C}^{n})$.
\end{example}
 
  \subsection{Defining equations of opposite Schubert varieties and their unions}
 
The defining ideals  $I(X^{w}) \subset \mathbb{C}[G/P_{k}]$ are, set-theoretically, defined by the quadrics defining  $G/P_{k}$ in $ \mathbb{P}(V(\omega_k))$ and the dual extremal Pl\"ucker coordinates given by

\begin{align}
\label{set}
 \left\{p_{\tau } : \tau \in W/W_{P_{k}}, \tau \ngeq w \right\}.
\end{align}
 
 If the representation $V(\omega_{k})$ is minuscule, then the extremal Pl\"ucker coordinates span $V(\omega_k)^*$ and \ref{set} define $X^w$ ideal-theoretically in $G/P_k$. We will however focus on a much more general picture, where we do not assume that the representations at hand are minuscule.

 Let $\tau \in W/W_{P_k}$ and let $ V(\tau \omega_k)^{\operatorname{opp}} \subset V(\omega_k)$ be the associated opposite Demazure module. We define it in our set-up as the $B^{-}$-submodule of $V(\omega_k)$ generated by $\tau v_{\omega_k}$, where $v_{\omega_k}$ is a highest weight vector. The inclusion embedding $ V(\tau \omega_k)^{\operatorname{opp}}  \hookrightarrow V(\omega_k)$ dually induces a  restriction map 
  $\operatorname{res}^{\tau}:V(\omega_k)^{*} \rightarrow {V(\tau \omega_k)^{\operatorname{opp}}}^{*}.$

 \begin{thm} \cite[Theorem 16]{lakshmibai2003richardson}
 \label{defining}
Scheme-theoretically, opposite Schubert varieties and their unions are defined linearly as subvarieties of $G/P$. In particular: let $r_{1}, \dots , r_{s} \in S^{2}(V(\omega_k)^{*})$ be the quadratic generators of the homogeneous vanishing ideal of $G/P \subset \mathbb{P}(V(\omega_k))$. Then $\left\{r_{1}, \dots , r_{s}\right\} \cup  \operatorname{ker}(res^{\tau})$  is a generating set for the homogeneous vanishing ideal of $X^{\tau} \subseteq \mathbb{P}(V(\omega_k))$. 
 \end{thm}

 \begin{thm}
 \label{definingunions}
The scheme-theoretic defining ideal of the union of two opposite Schubert varieties   $I(X^{\tau}\cup X^{\nu})$ is generated in degree one by the dual Pl\"ucker coordinates vanishing on both  $V(\tau \omega_k)^{\operatorname{opp}}$ and $V(\nu \omega_k)^{\operatorname{opp}}$. 
 \end{thm}
 
 \begin{proof}
 We refer the reader to \cite[Section 2, Section 5]{lakshmibai2003richardson} for the necessary background on standard monomials. Let $\tau, \tau' \in W/W_k$. By the proof \cite{lakshmibai2003richardson} of Theorem \ref{defining} we know that the kernel of the restriction map 
 \[\mathbb{C}[G/P_k] \longrightarrow \mathbb{C}[X^{\tau} \cup X^{\tau'}]\]

 \noindent has as basis the set of all standard monomials on $G/P_k$ that are neither standard on $X^{\tau}$ nor on $X^{\tau'}$. The indexing set of the factors of a standard monomial is linearly ordered. Now, a given standard monomial is standard on $X^{\tau} \cup X^{\tau'}$ if and only if its maximal factor is. Therefore,  standard monomial belongs to $I = I(X^{\tau} \cup X^{\tau'})$ if and only if the maximal factor of the monomial is in $I$. This implies that $I$ is generated by its degree one elements. 
 
 \end{proof}
 
 \begin{rmk}
 Theorems \ref{defining} and \ref{definingunions} hold more generally for Richardson varieties and unions thereof \cite{lakshmibai2003richardson}. In particular also for Schubert varieties. Also, the entire content of this section holds after replacing the fundamental weight $\omega_k$ with any dominant weight $\lambda$. We have, however, not included it as part of our notation since in this particular manuscript we only deal with fundamental weights and opposite Schubert varieties.
 \end{rmk}

\section{Proof of our main result}
 \subsection{Remarks on our notation}
 In the Introduction, we have introduced Pl\"ucker coordinates $p_{x_i}, p_{y_j}, p_{z_k}$ and $p_{u}$, labelled according to the pattern described by the graph $\mathcal{G}_{k}$ (see Figure \ref{graphgeny}). In terms of the notation introduced in the previous section, the equivalence is as follows. We have $k = x_{x_{c-2}}$ and each node of our graph $\mathcal{G}_{k}$ corresponds to an eleinment  $W/W_k$ as follows. If one walks along the graph starting at $p_{x_{c-2}}$ towards $p_{u}$, then the vertices correspond to the elements with representatives $id, s_{x_{c-2}}, s_{x_{c-1}}s_{x_{c-2}}, ..., s_{x_{1}}\cdots s_{x_{c-2}}$ and $s_{u}s_{x_{1}}\cdots s_{x_{c-2}}$. Then one walks along each arm and gets elements $s_{y_{1}}s_{u}s_{x_{1}}\cdots s_{x_{c-2}},..., s_{y_{d}}\cdots s_{y_{1}}s_{u}s_{x_{1}}\cdots s_{x_{c-2}}$ on the left arm, and $s_{z_{1}}s_{u}s_{x_{1}}\cdots s_{x_{c-2}},..., s_{z_{t}}\cdots s_{y_{1}}s_{u}s_{x_{1}}\cdots s_{x_{c-2}}$ on the right arm, respectively.

 To each of these elements $w$ corresponds a Pl\"ucker coordinate $p_w$, according to the notation introduced in this section. We will however abuse this notation and denote by $p_{l}$ the Pl\"ucker coordinate corresponding to one of the elements mentioned in the previous paragraph written as $v = s_{l} w$. This coincides with the notation introduced in Figure \ref{graphgeny}. Similarly, we will denote by $X^{l}$ the corresponding opposite Schubert varieties. We write $Y^l = X^l \cap C^{id}$. Since by definition there are no repetitions in the labelings of the edges of $\mathcal{G}_{k}$, this notation is well-defined.


\begin{proof} [Proof of Theorem \ref{ourthm}]
First we prove 2. and 3. assuming 1. Then we will show 1.  Now, since the ideals $I(X^{y_l})$, resp. $I(X^{z_m})$ are prime, their heights are respectively $l+c$ and $m+c$, that is, the codimensions of the varieties they define. Now, by construction, 
we have $c = \operatorname{ht}(I(X^{y_1})) = \operatorname{ht}(I(X^{z_1})) \leq l+c =  \operatorname{ht}(I(X^{y_l})) , m+c =  \operatorname{ht}(I(X^{z_m}))$. Now, let $I = I(X^{z_1})$ (resp. $ I = I(X^{y_1})$) and $J = I(X^{y_l}))$ (resp. $J = I(X^{z_m})$). Then, by Krull's altitude theorem, we know that:
\[\operatorname{ht}(J) \leq \mu (I \cap J),\]

\noindent
where for an ideal M, recall that $\mu(M)$ denotes the minimal number of generators of $M$. By Remark \ref{geometricri} it remains to show that equality holds, so we need to study the intersection $I \cap J$, that is, the defining ideal of $X^{z_1} \cup X^{y_l}$ (resp.  $X^{y_1} \cup X^{z_m})$). Now, by Theorem \ref{definingunions}, we know that $I( X^{z_1} \cup X^{y_l})$ (resp.  $I(X^{y_1} \cup X^{z_m}))$) is generated by the dual Pl\"ucker coordinates which vanish on both $X^{z_1}$ and $X^{y_l})$ (resp. on both $X^{y_1}$ and $X^{z_m})$).  Now, by Theorem \ref{definingunions} we know that $I(X^{z_1})$ (resp. $I(X^{y_1})$) is generated in degree 1 by the dual Pl\"ucker coordinates $p^{*}_{\emptyset},\ldots ,p^{*}_{u}, \ldots p^{*}_{y_1} \ldots p^{*}_{y_d}$ (resp. $p^{*}_{\emptyset},\ldots ,p^{*}_{u}, \ldots p^{*}_{z_1} \ldots p^{*}_{z_t}$). Out of these,  only $p^{*}_{\emptyset},\ldots ,p^{*}_{u}, \ldots p^{*}_{y_1} \ldots p^{*}_{y_{l-1}}$ vanish on $X^{y_l}$ (resp. only $p^{*}_{\emptyset},\ldots ,p^{*}_{u}, \ldots p^{*}_{z_1} \ldots p^{*}_{z_{m-1}}$ vanish on $X^{z_m}$). Therefore we conclude that the ideal $I(X^{y_l}))$ (resp. $I(X^{z_m})$) is a residual intersection of $I(X^{z_1})$ (resp. $I(X^{y_1})$). \\

Now we proceed to prove 1. Consider the crystal graph $B(\omega_{k})$.  The graph $\mathcal{G}_{k}$ is embedded (as a coloured, directed graph) into $B(\omega_{k})$ in such a way that $p_{\emptyset}$ corresponds to the lowest weight vertex of $B(\omega_{k})$ (of weight $-\omega_{k}$). The coordinates $p_{y_1}, p_{z_1}$ correspond to vertices in  $B(\omega_{k})$ with weights of the form $\sum a_{i} \omega_{i}$, where precisely one coefficient $a_{i}$ is positive: the one corresponding to $i = y_1$, respectively $i = z_1$. Let $\mathfrak{g}_{j}$ (for $j = y_1,z_1$) be the Levi subalgebra of $\mathfrak{g}$ corresponding to the sub-diagram $D_{j}$ of $D$ obtained by deleting the vertex corresponding to $j$. The coordinates $p_{y_1}, p_{z_1}$ each correspond to the lowest weight vertex of an irreducible $\mathfrak{g}_{j}$-crystal.  In fact the opposite Demazure module needed by Theorem \ref{defining} to obtain the defining equations for $X^{y_1}$, respectively $X^{z_1}$, is the $\mathfrak{b}$-module generated by $p_{y_1}$, respectively $p_{z_1}$, and this module coincides with the module generated by the Lie subalgebra of $\mathfrak{g}$ generated by $\mathfrak{b}$ together with $\mathfrak{g}_{j}$. If we consider the irreducible components of $\operatorname{res}^{\mathfrak{g}}_{\mathfrak{g}_{j}}(V(\omega_{k}))$, we have precisely one component whose dual vanishes on the opposite Demazure module $V({w_j}\omega_{k})^{\operatorname{opp}}$. Also, note that the corresponding component in the crystal graph corresponds precisely to the nodes in the arm (in  $\mathcal{G}_{k}$ embedded in $B(\omega_{k})$) opposite to $p_j$. This is due to the fact that the Levi sub-algebra $\mathfrak{g}_{j}$ is a product of two Lie algebras of type A, and the irreducible component at hand is therefore isomorphic to an irreducible $\mathfrak{sl}(r,\mathbb{C})$-crystal of highest weight either $\omega_{1}$ or $\omega_{r}$. Note that this does not happen if our initial node $k$ is not extremal (or minuscule). Therefore these are precisely the nodes whose corresponding dual coordinates vanish on $X^{j}$. This concludes the proof.

\end{proof}

\section{Examples and explicit computations in the minuscule cases}
In order to perform our explicit computations, we will, in this section, consider the intersections of our opposite Schubert varieties at hand with the big open cell, $C^{id}$. In the minuscule cases we are able to compute the defining equations of these intersections explicitly. We provide alternative proofs of Theorem \ref{ourthm}, using this information, after replacing opposite Schubert varieties with their respective intersections with the big open cell. In the exceptional cases, these are obtained using Macaulay2 and the material from \cite{FTW20}.

\subsection{Residual intersections in type $A_{n-1}$}\label{sec:typea}
In this subsection we focus on Schubert varieties in Grassmannians $\op{Gr}(k,n)$, for $1\leq k \leq n-1$. Without loss of generality let us also assume $k\le n-k$, otherwise we pass to the dual Grassmannian.

The set $W/W_{P_{k}}$ can be identified with the subsets $K$ of cardinality $k$ of $[1,n]$. To each such subset is associated an opposite Schubert variety denoted by $X^{K}$. The graph $\mathcal{G}_{k}$ has the left arm consisting of the varieties $X^{\{1,2,\ldots, k-1, k+s\}}$ for $s=0,\ldots, n-k$; the right arm consists of varieties $X^{[1,k+1]\setminus \lbrace k+1-s\rbrace }$ for $s=0,1,\ldots, k$.

The big open cell can be identified with the set of matrices

\[ M = \left(\begin{matrix} y_{1,1}& y_{1,2}&\ldots &y_{1,n-k}&1&0&\ldots &0\\
y_{2,1}&y_{2,2}&\ldots&y_{2,n-k}&0&1&\ldots&0\\
\ldots&\ldots&\ldots&\ldots&\ldots&\ldots&\ldots&\ldots&\\
y_{k,1}&y_{2,2}&\ldots&y_{k,n-k}&0&0&\ldots&1
\end{matrix}\right).\]

\noindent
The ideal \[I(1,2,\ldots, k-1, k+s) := I(Y^{\{1,2,\ldots, k-1, k+s\}})\] \noindent can be identified with the ideal of maximal minors
of the $(k-s+1)\times k$ submatrix of the matrix $M$ consisting of last $k-s+1$ rows. The ideal \[I([1,k+1]\setminus\lbrace k+1-s\rbrace ):= I(Y^{[1,k+1]\setminus \lbrace k+1-s\rbrace })\] \noindent is the 
ideal of maximal minors of the first $k-s+1$ columns of $M$. The equations equivalent to Theorem \ref{ourthm} are the following:
\begin{align*}
(p_{1,2,\ldots ,k},\ldots ,p_{1,2,\ldots ,k-1, k+s}):I(1,2,\ldots ,k-2, k, k+1)=I(1,2,\ldots ,k-1,k+s+1)\\
(p_{[1,k+1]\setminus\lbrace k+1\rbrace},\ldots ,p_{[1,k+1]\setminus \lbrace k+1-s\rbrace}: I(1,2,\ldots,k-1,k+2)=I([1,k+1]\setminus \lbrace k+1-s-1\rbrace ).
\end{align*}
The second set of equations was proved by Huneke \cite{H83}. To prove the second set of equations we first show that
\[(p_{[1,k+1]\setminus\lbrace k+1\rbrace},\ldots ,p_{[1,k+1]\setminus \lbrace k+1-s\rbrace}: I(1,2,\ldots,k-1,k+2)\supset I([1,k+1]\setminus \lbrace k+1-s-1\rbrace ).\]
This is equivalent to saying that
\[I(1,2,\ldots,k-1,k+2) I([1,k+1]\setminus \lbrace k+1-s-1\rbrace )\subset (p_{[1,k+1]\setminus\lbrace k+1\rbrace},\ldots ,p_{[1,k+1]\setminus \lbrace k+1-s\rbrace}\]
\noindent This means that
\[p_Ip_J\in \subset (p_{[1,k+1]\setminus\lbrace k+1\rbrace},\ldots ,p_{[1,k+1]\setminus \lbrace k+1-s\rbrace}\]
for $p_I\in  I(1,2,\ldots,k-1,k+2)$, $p_J\in  I([1,k+1]\setminus \lbrace k+1-s-1\rbrace )$.
The dual Pl\"ucker coordinate $p_I$ is a $k\times k$ minor on all rows of $Y$ and the first $k+1$ columns.
The dual Pl\"ucker coordinate $p_J$ is a $(k-s+1)\times (k-s+1)$ minor on the last $k-s+1$ rows of $Y$. The other implication follows from an argument of Huneke \cite{H83}.

\begin{example}

Consider the case $k = 2$. Proposition \ref{g2n} describes the pattern in this case. In Figure \ref{26}, we see the graph $\mathcal{G}_{2}$ embedded into the Bruhat graph for $W/W_{P_{2}}$ which coincides with the crystal graph for the minuscule $\mathfrak{sl}(6,\mathbb{C})$ representation $V(\omega_2) = \bigwedge^{2}\mathbb{C}^{6}$. The graph $\mathcal{G}_{2}$ is the subgraph highlighted in bold red spanned by the vertices $p_{12} = p^{*}_{\emptyset}, p_{13} = p^{*}_{2}, p_{23} = p^{*}_{1}, p_{14} = p^{*}_{3}, p_{15} = p^{*}_{4}, p_{16} = p^{*}_{5}$, where our identification is intended to match the notation used to formulate Theorem \ref{ourthm}.

\begin{prp}
\label{g2n}
Let $j< n$ and 
\begin{align*}
I &= (p_{in}: 1 \leq i < n) \\
K_{j} &=  (p_{in}: j \leq i < n) \\
I_{j} &=  (p_{st}: j \leq s < t \leq n).
\end{align*}
\noindent Then $K_{j}: I = I_{j}$.
\end{prp}

\begin{proof}
The Plücker relation $p_{st}p_{in} = p_{si}p_{tn} - p_{sn}p_{tn}$ implies $K_{j}: I \supseteq I_{j}$. The same relation implies that a given dual Plücker coordinate $p_{st}$ belongs to the colon ideal $K_{j}:I$ if and only if $p_{st} \in I_{j}$.
\end{proof}

\begin{figure}
\includegraphics[scale= 0.9]{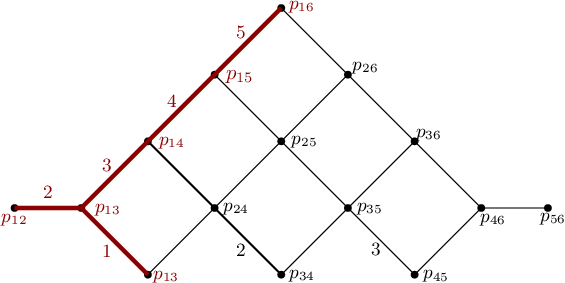}
\caption{Bruhat graph for $W/W_{P_{2}}$ for $A_{5}$}
\label{26}
\end{figure}

\end{example}

\subsection{Residual intersections of Pfaffian ideals and Schubert Varieties of type $D_{m}$}\label{sec:typed}
\label{minusculeD}

\subsubsection{The half-spin representation $V(\omega_{n})$}
Let  $P_{n} \subset \op{SO}(2n,\mathbb{C})$ be a parabolic subgroup which stabilizes the minuscule fundamental weight $\omega_{n}$, using Bourbaki notation to label our fundamental weights. It is known  that $ \op{SO}(2n,\mathbb{C})/P_{n}$ is naturally embedded into $\mathbb{P}(V(\omega_{n}))$, where $V(\omega_{n})$ is one of the half-spin representations. It is also isomorphic to one of two connected components of the \textit{orthogonal} or \textit{isotropic} Grassmannian $\op{IG}(2n,\mathbb{C})$ of maximal subspaces of an even-dimensional vector space which are isotropic with respect to a non-degenerate, symmetric bilinear form. Since the fundamental weight $\omega_{n}$ is minuscule, the Bruhat graph for the quotient $W/W_{P_{n}}$ coincides with the crystal graph associated to the half-spin representation $V(\omega_{n})$.

The crystal $B(\omega_{n})$ for the half-spin representation $V(\omega_n)$ is defined as follows (we refer the reader to \cite{KashiwaraNakashima1994} for a full treatment of this topic for classical Lie algebras).
Its elements are parametrized by sequences $\underline{i} = (i_{1},...,i_{n}), i_{j} = \pm $ of plus and minus signs such that their total product equals one. Moreover, 

\[
e_{j}(i_{1},...,i_{n})= 
\begin{cases}
(i'_{1},...,i'_{n}), i'_{j} = +, i'_{j+1} = -, i'_{k} = i_{k} \quad \forall  k \neq j \hbox{ if } i_{j} = - \hbox{ and } i_{j+1} = + \\
0 \hbox{ otherwise }
\end{cases}
\]

\[
f_{j}(i_{1},...,i_{n})= 
\begin{cases}
(i'_{1},...,i'_{n}) , i'_{j} = -, i'_{j+1} = +, i'_{k} = i_{k} \quad \forall k \neq j\hbox{ if } i_{j} = + \hbox{ and } i_{j+1} = - \\
0 \hbox{ otherwise }
\end{cases}
\]

\noindent
To each sequence $\underline{i}$ we can associate a minor of a generic $2n\times 2n$ skew-symmetric matrix by choosing the rows and columns corresponding to the indices of the minus signs in $\underline{i}$. The Pfaffians of these naturally index a basis of the half-spin representation $V(\omega_{n})$. In Figure \ref{d5} the reader can see the crystal graph for $B(\omega_{5})$ and the graph $\mathcal{G}_{5}$ embedded in it (in bold red).

\begin{figure}
\includegraphics[scale= 0.9]{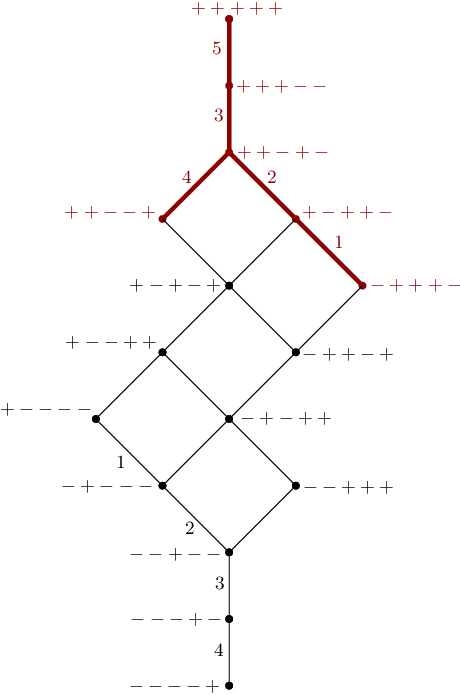}
\caption{Crystal graph for the half-spin representation $V(\omega_{5})$ of $\op{SO}(10,\mathbb{C})$}
\label{d5}
\end{figure}

\subsubsection{ The isotropic Grassmannian and its Schubert varieties}

Let $V$ be an even-dimensional complex vector space of dimension $2m$, and let 
\[\< -,-\rr: V \times V \rightarrow \mathbb{C}\]
be a symmetric, non-degenerate bilinear form on $V$. We say that a subspace $W \subset V$ is \textit{isotropic}  if $\< v,w\rr = 0$ for any $v,w \in W$. By Witt's theorem \cite{witt}, any such subspace can have dimension at most $n$. The subvariety $IG(m,2m) \subset \op{Gr}(m,2m)$ of $m$-dimensional isotropic subspaces of $V$ is called the \textit{isotropic Grassmannian} and has two connected components $S^{+}$ and $S^{-}$, isomorphic to $\op{SO}(2m,\mathbb{C})/P_{\omega_{m}}$ and  $\op{SO}(2m,\mathbb{C})/P_{\omega_{m-1}}$, respectively; these in turn can be canonically embedded into $\mathbb{P}(V(\omega_{m}))$ and $\mathbb{P}(V(\omega_{m-1}))$, respectively. Each point in the opposite big open subset of $S^{+}$ is generated by the rows of a matrix of the form $(I_{m} | X)$, where $I_{m}$ is the $m\times m$ identity matrix and $X$ is a skew-symmetric matrix. (Opposite) Schubert cells are intersections of the (opposite) Schubert cells in $\op{Gr}(m,2m)$ with $\op{SO}(2m,\mathbb{C})/P_{\omega_{m}}$ and hence are indexed by certain partitions. Since the crystal graph coincides with the Bruhat graph in our minuscule case we may associate a (an opposite) Schubert cell to each vertex $\underline{i} = (i_{1},...,i_{m})$ in $B(\omega_{m})$. Let $k_{1} < \cdots < k_{s}$ be the indices such that $i_{k_{t}}$ is a minus sign for $1\leq t \leq s$, and let $j_{1} < \cdots < j_{u}$ be the indices such that $i_{j_{t}}$ is a plus sign for $1 \leq t \leq u$. Recall that (opposite) Schubert cells and therefore (opposite) Schubert varieties in $\op{Gr}(m,2m)$ are labelled by sequences of indices $\underline{l} = 1 \leq l_{1} < \cdots < l_{n} \leq 2m$.

%

 The plus minus sequence $\underline{i}$ defines the index sequence

\[\underline{l}(\underline{i}) =  j_{1},\cdots, j_{u}, 2m - k_{s} +1, \cdots 2m - i_{1} - 1.\]

The following result is well-known.

\begin{thm}\cite{brownlakshmibai}
Let $X^{\underline{i}} \subset \op{SO}(2n,\mathbb{C})/P_{n}$ be the opposite Schubert variety associated to $\underline{i}$ and let $\Sigma^{\underline{l}(\underline{i})} \subset \op{Gr}(n,2n)$ be the opposite Schubert variety in the Grassmannian corresponding to the index sequence $\underline{l}(\underline{i})$. Then 

\[ \Sigma^{\underline{l}(\underline{i})} \cap \op{SO}(2n,\mathbb{C})/P_{n} = X^{\underline{i}}. \]

Moreover, all Schubert varieties in $\op{SO}(2n,\mathbb{C})/P_{n}$ arise in this way.
\end{thm}

Let $\underline{i} = (i_{1}, \cdots, i_{m})$ be a vertex in $B(\omega_{m})$ as described above, that is, it is a size $n$ sequence of plus and minus signs whose product is positive. In this section we will describe how to obtain equations for the intersection $Y^{\underline{i}}$. Recall our notation from above and consider the pfaffians $Pf(k_{1}\cdots k_{s})$ of the minors of an $m \times m$ generic skew-symmetric matrix $A$ defined by picking all the rows and columns indexed by $k_{1},\cdots, k_{s}$. These are polynomials in the variables $\left\{ x_{ij}\right\}_{1 \leq i < j \leq m}$. We will denote by ${p_f}_{\underline{i}}$ the pfaffian $Pf(k_{1}\cdots k_{s})$. Recall that since we have identified the Bruhat graph with the crystal graph we may write $\underline{i } \leq \underline{i'}$ whenever the order relation is preserved for minimal length representatives of $\underline{i}$ and $\underline{i'}$. 

\begin{prp}
\label{equations}
The defining ideal $I^{\underline{i}}$ of the intersection $Y^{\underline{i}}$ is generated by the pfaffians ${p_f}_{\underline{i'}}$ for $\underline{i'} \nleq  \underline{i}$.
\end{prp}

\subsubsection{Residual intersections of pfaffian ideals}

\begin{thm}
Let $A$ be a generic $(2n+1)\times (2n+1)$ skew-symmetric matrix of variables, let $R = \mathbb{C}[A]$ and let $I \subset R$ be the ideal generated by the sub-maximal pfaffians of $A$. Let $3 \leq j \leq m$ and $X_{j}$ be the top left-aligned $m-j \times m-j$ minor of $A$. Let $Pf_{j}(A)$ be the ideal generated by the pfaffians of all the minors of $A$ containing $A_{j}$. Then:

\begin{align}
\label{oddcolond}
I \left( \left\{ {p_f}_{[1,2n+1]/ \left\{ t \right\} } \right\}_{t = j}^{2n+1} \right) : I = Pf_{j}(A)
\end{align}
\quad

\noindent Moreover, the ideal $Pf_{j}(A)$ is the defining ideal of an open subset of the Schubert variety in $\op{SO}(2m)/ \omega_{m}$, where $m = 2n+1$, corresponding to the node indexed by the $\pm$ sequence consisting of one single $+$ in place $ m-j$, and the rest minuses. Additionally, Let $A'$ be the matrix obtained from $A$ by setting the south-eastern aligned $j\times j$ to zero. Then we have $Pf_{j}(A) = Pf_{j}(A')$.

\end{thm}

\begin{proof}
Note that by Proposition \ref{equations}, the ideal $Pf_{j}(A)$ describes an open subset of the desired Schubert variety. This follows from the crystal graph description of the half-spin representations using $\pm $ sequences (recall that in our set-up a - in the sequence means the column is not omitted) We proceed to describe the Kustin--Ulrich generators associated to $A$ and $j$ described in \cite[Definition 4.10]{KU}. For $t \in [1,m]$, let $B$ be the $m \times j$ matrix with columns of the form $(a_{t1},..., a_{tm})$, where 
\
\[ p_{[1,m]/ \left\{ t\right\} } = \sum_{i = 1}^{m} a_{ti} \hbox{ }p_{[1,m]/ \left\{ i\right\} }  \]

\noindent for $j \leq t \leq m$. That is, the matrix $B$ is of the form $\frac{Z_{m-j \times j}}{Id_{j}}$, where $Z_{m-j \times j}$ is the zero matrix of size $m-j \times j$, and $Id_{j}$ is the $j \times j$ identity matrix. Theorem 10.2 in \cite{KU} states that if one considers the matrix  

\
\[ T = \begin{pmatrix} A & B \\ -B^{t} & Z_{j \times j} \end{pmatrix}, \]

\noindent
then the pfaffians of the even-sized minors of $T$ containing $A$ generate the residual intersection ideal \ \[I \left( \left\{ {p_f}_{[1,2n+1]/ \left\{ t \right\} } \right\}_{t = j}^{2n+1} \right) : I  \subset R.\]
However, it is easy to check by direct computation that these pfaffians are precisely the pfaffians of the minors of $A$ which contain the top left-aligned $j \times j$ minor  $A_{j}$  of $A$, and these pfaffians are also the same as the pfaffians of $A'$ which contain the top-left aligned minor $A'_j = A_j$.  
\end{proof}

\subsubsection{Finite free resolutions}

We deal with the case $c=3$, $d=n-3$, $t=1$. We are interested in the varieties $Y^{y_1}$, $Y^{y_2}$, \ldots , $Y^{y_d}$. The finite free resolutions of these varieties were described by Kustin and Ulrich  in \cite{KU}. They treated them as generic residual intersections of the ideal of codimension $3$ Pfaffians of an $(2m+1)\times (2m+1)$  skew-symmetric matrix. But by our main result these are exactly Schubert varieties $Y^{y_1}$, \ldots , $Y^{y_d}$. Reading the paper \cite{KU} is not easy because of many technical details. We provide a more geometric approach. It does not give differentials in our free resolutions, but it gives a description of Betti tables of our Schubert varieties.
As an example we give the resolutions of complexes constructed by Kustin and Ulrich in \cite{KU}.

 Let $U$, $V$ be  vector spaces of dimension $u$, $v$ respectively.

We use the geometric method of calculating syzygies \cite{JWm03}. We work over the polynomial ring $A=Sym(\bigwedge^2 V^*\oplus U\otimes V^*)$.

We think of the spectrum of $A$ as a space of skew-symmetric $(v+u)\times (v+u)$ matrices

$$\left(\begin{matrix}\phi&\psi\\-\psi^t&0
\end{matrix}\right).$$

\begin{prp}\label{prop:KU}
\begin{enumerate}
\item Let $dim\ V=2h$, consider Grassmannnian $Grass(2h-1,  V)$ with tautological sequence
$$0\rightarrow {\mathcal R}\rightarrow V\times Grass(2h-1, V)\rightarrow{\mathcal Q}\rightarrow 0,$$
$rk\ {\mathcal R}=2h-1$, $rk\ {\mathcal Q}=1$.
Consider the subbundle

$$\xi ={\mathcal Q}^*\otimes{\mathcal R}^*\oplus U\otimes {\mathcal Q}^*$$
in the trivial bundle with fibre $\bigwedge^2 V^*\oplus U\otimes V^*$.
We have 
$$\bigwedge^p\xi = \oplus_{i+j=p} S_i{\mathcal Q}^*\otimes\bigwedge^i {\mathcal R}^*\otimes\bigwedge^j U\otimes S_j {\mathcal Q}^*.$$

We get the complex $(D^0(\rho)_\bullet ,d)$ constructed in \cite{KU} for the pair $(u,2h)$.
Since we work on the projective space the resulting complex is really characteristic free, where in terms of the top cohomology groups we need to use Weyl functors instead of Schur functors.

\item Let $dim\ V=2h+1$, consider Grassmannnian $Grass(2h,  V)$ with tautological sequence
$$0\rightarrow {\mathcal R}\rightarrow V\times Grass(2h, V)\rightarrow{\mathcal Q}\rightarrow 0,$$
$rk\ {\mathcal R}=2h$, $rk\ {\mathcal Q}=1$.
Consider the subbundle

$$\xi ={\mathcal Q}^*\otimes{\mathcal R}^*\oplus U\otimes {\mathcal Q}^*$$
in the trivial bundle with fibre $\bigwedge^2 V^*\oplus U\otimes V^*$.
We have 
$$\bigwedge^p\xi = \oplus_{i+j=p} S_i{\mathcal Q}^*\otimes\bigwedge^i {\mathcal R}^*\otimes\bigwedge^j U\otimes S_j {\mathcal Q}^*.$$

We get the complex $(D^0(\rho)_\bullet ,d)$ constructed in \cite{KU} for the pair $(u,2h+1)$.
Since we work on the projective space the resulting complex is really characteristic free, where in terms of the top cohomology groups we need to use Weyl functors instead of Schur functors.
\end{enumerate}
\end{prp}
\begin{proof} The idea is to push down the Koszul complex  and use the Borel--Weil--Bott theorem to calculate the relevant cohomology. Let us start with the case of $dim\ G=2h$. Pushing down $\bigwedge^p\xi$ we get the weights
$(-p, 0^{v-1-i}, -1^{i})$. They give a nonzero contribution in the cases of $p=0$ (contribution to $H^0$), $p=v-i$, $0\le i\le g-1$ (contribution to $H^{v-i-1}(\bigwedge^{v-i}\xi)$) and $p\ge v+1$ when they give a contribution to top cohomology $H^{v-1}(\bigwedge^p\xi)$. Notice that for the terms coming from middle homologies occur in the first term of our complex $\FF_\bullet$ and the top homology terms come in terms $\FF_j$ with $j\ge 2$. Notice also that for terms on $\FF_1$ we have $i+j=v-i$, so $j=v-2i$ is even. So we get a complex
$\FF_\bullet$ with
$$\FF_s =\oplus_{i+j=s+v-1, 0\le i\le v-1} S_{(-1)^{v-1-i}, (-2)^i, -s)}G\otimes\bigwedge^j U\otimes A(-s-v+1),$$
$$\FF_1=\oplus_{0\le i\le v-1, i\ \rm{even}} S_{(-1)^v}G\otimes\bigwedge^i U\otimes A(-v+i),$$
$$\FF_0=A.$$

Notice that our complex is characteristic free because the terms from $\FF_1$ are just tensor products of exterior powers, and the terms $\FF_s$ with $s\ge 2$ come from the top cohomology groups of irreducible homogeneous bundles on the projective space, so they are Weyl functors. In conclusion,  this complex  has length $u$ as the last term occurs for $i=v-1, j=u$. It is a free resolution of the $A$-module $A/J(v,u)$ for a perfect ideal $J(v,u)$ of codimension $u$.

For the case $dim\ V=2h+1$ the reasoning is practically the same. The terms are again
$(-p, 0^{v-1-i}, -1^{i})$. They give a nonzero contribution in the cases of $p=0$ (contribution to $H^0$), $p=v-i$, $0\le i\le v-1$ (contribution to $H^{v-i-1}(\bigwedge^{v-i}\xi)$) and $p\ge v+1$ when they give a contribution to top cohomology $H^{v-1}(\bigwedge^p\xi)$. Notice that for the terms coming from middle homologies occur in the first term of our complex $\FF_\bullet$ and the top homology terms come in terms $\FF_j$ with $j\ge 2$. Notice also that for terms on $\FF_1$ we have $i+j=v-i$, so $j=v-2i$ is odd. So we get a complex
$\FF_\bullet$ with
$$\FF_s =\oplus_{i+j=s+v-1, 0\le i\le v-1} S_{(-1)^{v-1-i}, (-2)^i, -s)}V\otimes\bigwedge^j U\otimes A(-s-v+1),$$
$$\FF_1=\oplus_{0\le i\le v-1, i\ \rm{odd}} S_{(-1)^v}V\otimes\bigwedge^iU\otimes A(-v+i),$$
$$\FF_0=A.$$

  This complex also has length $u$ as the last term occurs for $i=v-1, j=u$. It is a free resolution of the $A$-module $A/J(v,u)$ for a perfect ideal $J(v,u)$ of codimension $u$.

\end{proof}

\begin{rmk}

In both cases one can see easily that we get the free resolution of the Schubert variety $Y^{y_{u-2}}$ intersected with the linear subspace given by $u\choose 2$ conditions making the right lower corner of the matrix
$$\left(\begin{matrix}\phi&\psi\\-\psi^t&0
\end{matrix}\right)$$
equal to zero. Codimensions of both ideals are the same, so the intersection is transversal so the Betti table of the resolutions of  $I(Y^{y_{u-2}})$ and our ideal $J(v,u)$  are the same.
\end{rmk}

\subsection{Residual intersections for the minuscule exceptional cases}

The main reference for this section is \cite{FTW20}. As for the whole of this section, to perform our explicit computations we always consider the intersection of our opposite Schubert varieties with the big open cell. Let $R = \mathbb{C}[G/P_k \cap C^{w_0}]$

\subsubsection{Residual intersections in $E_6/P_1$}\label{sec:typee6}
We have  $c = 4, d=1, t = 2$ and  $x_2 = 6 ,x_1 = 5,u = 4, y_1 = 2, z_1 = 3, z_2 = $.
Up to codimension $3$ there is one Schubert variety per codimension and each is a complete intersection. In codimension $4$ there are two subvarieties, an almost complete intersection $Y^{2}$ and a hyperplane section $Y^{3}$ of the codimension $3$ Gorenstein ideal generated by the maximal Pfaffians of a $5\times 5$ skewsymmetric matrix (with linear entries in $R$). We have $I_{2} := I(Y^{2}) = (\bar{z}_1,\bar{z}_2,\bar{z}_3,\bar{z}_4, \bar{z}_{5})$, and 
$I_{3} := I(Y^{3}) =  (\bar{z}_1,\bar{z}_2,\bar{z}_3,\bar{z}_4,z_5,y_{1234})$. The minimal free resolution of $Y^{2}$ is of the form:
\[ 
0 \rightarrow R(-7) \oplus R(-6) \rightarrow R^{10}(-5) \rightarrow R^{11}(-4) \oplus R(-3) \rightarrow R^5(-2) \rightarrow R.
\]
The minimal free resolution of $Y_{3}$ is of the form:
\[
0 \rightarrow R(-6) \rightarrow R^5(-4) \oplus R(-5) \rightarrow R^5(-3) \oplus R^5(-3) \rightarrow R(-1) \oplus R^5(-2) \rightarrow R. 
\]

Furthermore, the ideals of these varieties are linked by $a_1 = (p_{\emptyset},p_{6},p_{5},p_4)$. Indeed, by the theory of linkage \cite{ulrich} the dual of the reduced mapping cone of the resolution of the complete intersection to the resolution of $Y_{2}$ gives a resolution for $Y_{3}$ (up to a grading shift). Let $a_2 = (p_{\emptyset},p_{6},p_{5},p_4,p_{2})$. Then $ I_{1} := I(Y^{1}) = a_{2}: I_{3}$. In the notation of \cite{FTW20}, we have $I_{3} = J_{23}, I_{2} = J_{22}, I_{1} = J_{17}.$ Below include some computations in Macaulay2, where we also use the notation from \cite{FTW20}. The translation between our notation and \cite{FTW20} is given by: 

\SaveVerb{z5}=z_5= 
\SaveVerb{zb5}=zb_5= 
\SaveVerb{zb4}=zb_4= 
\SaveVerb{zb3}=zb_3= 
\SaveVerb{zb2}=zb_2= 
\SaveVerb{zb1}=zb_1= 
\SaveVerb{y}=y_1234=

\[
 \UseVerb{zb1}= {p_{\emptyset}}^{*}, \UseVerb{zb2} = {p_{6}}^{*}, \UseVerb{zb3} = {p_{5}}^{*},
  \UseVerb{zb4} = {p_{4}}^{*},  \UseVerb{zb5} = {p_{2}}^{*}, \UseVerb{z5} =  {p_{3}}^{*},
  \UseVerb{y} = {p_{1}}^{*}.
\]

The rest of the variables in the calculations below do not appear as vertices of the graph $\mathcal{G}_{6}$.

%
%

\begin{figure}[h]
\label{graphe6p1}
\includegraphics[scale = 0.8]{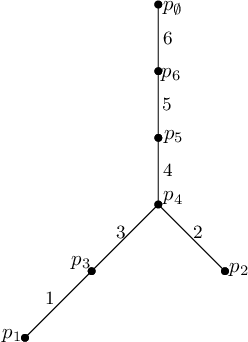}
\caption{The graph $\mathcal{G}_{6}$ for $E_{6}$. }
\end{figure}

\begin{verbatim}


 
R = QQ[y_0,y_12,y_13,y_14,y_15,y_23,y_24,y_25,y_34,y_35,y_45,y_1234,y_1235,y_1245,y_1345,y_2345]

zb_1 = y_15*y_1234 - y_14*y_1235 +y_13*y_1245 - y_12*y_1345
zb_2 = y_25*y_1234 - y_24*y_1235 +y_23*y_1245 - y_12*y_2345
zb_3 = y_35*y_1234 -y_34*y_1235 +y_23*y_1345 - y_13*y_2345
zb_4 = y_45*y_1234 - y_34*y_1245 +y_24*y_1345 -y_14*y_2345
zb_5 = y_45*y_1235 - y_35*y_1245 + y_25*y_1345- y_15*y_2345
z_5 = y_0*y_1234- y_34*y_12 +y_24*y_13 -y_23*y_14
z_4 =y_0*y_1235 - y_35*y_12 +y_25*y_13 -y_15y_23
z_3 = y_0*y_1245 -y_45*y_12 +y_25*y_14 - y_24*y_15
z_2 = y_0*y_1345 - y_45*y_13 +y_35*y_14 - y_34*y_15
z_1 = y_0*y_2345 - y_45*y_23 +y_35*y_24 - y_34*y_25

J23 = ideal(zb_1, zb_2, zb_3, zb_4, zb_5, y_1234)    
      = ideal(y_15*y_1234 - y_14*y_1235 +y_13*y_1245 - y_12*y_1345,
                  y_25*y_1234 - y_24*y_1235 +y_23*y_1245 - y_12*y_2345,
                  y_35*y_1234 -y_34*y_1235 +y_23*y_1345 - y_13*y_2345,
                  y_45*y_1234 - y_34*y_1245 +y_24*y_1345 -y_14*y_2345,
                   y_0*y_1234- y_34*y_12 +y_24*y_13 -y_23*y_14,y_1234)

J22 = ideal( y_15*y_1234 - y_14*y_1235 +y_13*y_1245 - y_12*y_1345,
                   y_25*y_1234 - y_24*y_1235 +y_23*y_1245 - y_12*y_2345,
                   y_35*y_1234 -y_34*y_1235 +y_23*y_1345 - y_13*y_2345,
                    y_45*y_1234 - y_34*y_1245 +y_24*y_1345 -y_14*y_2345,]
                    y_45*y_1235 - y_35*y_1245 + y_25*y_1345- y_15*y_2345)
 
a_1 = ideal( y_15*y_1234 - y_14*y_1235 +y_13*y_1245 - y_12*y_1345,
                     y_25*y_1234 - y_24*y_1235 +y_23*y_1245 - y_12*y_2345, 
                     y_35*y_1234 -y_34*y_1235 +y_23*y_1345 - y_13*y_2345, 
                     y_45*y_1234 - y_34*y_1245 +y_24*y_1345 -y_14*y_2345)
 
 i59 : a_1:J22 == J23
o59 = true

i60 : a_1:J23 == J22
o60 = true

 J17 = ideal(y_15*y_1234 - y_14*y_1235 +y_13*y_1245 - y_12*y_1345, 
                   y_25*y_1234 - y_24*y_1235 +y_23*y_1245 - y_12*y_2345, 
                   y_35*y_1234 -y_34*y_1235 +y_23*y_1345 - y_13*y_2345,
                   y_45*y_1234 - y_34*y_1245 +y_24*y_1345 -y_14*y_2345,
                   y_45*y_1235 - y_35*y_1245 + y_25*y_1345- y_15*y_2345, 
                   y_0*y_1234- y_34*y_12 +y_24*y_13 -y_23*y_14,
                   y_0*y_1235 - y_35*y_12 +y_25*y_13 -y_15*y_23, 
                  y_0*y_1245 -y_45*y_12 +y_25*y_14 - y_24*y_15,
                 y_0*y_1345 - y_45*y_13 +y_35*y_14 - y_34*y_15, 
                 y_0*y_2345 - y_45*y_23 +y_35*y_24 - y_34*y_25)

a_2 =  ideal( y_15*y_1234 - y_14*y_1235 +y_13*y_1245 - y_12*y_1345,
 y_25*y_1234 - y_24*y_1235 +y_23*y_1245 - y_12*y_2345,
  y_35*y_1234 -y_34*y_1235 +y_23*y_1345 - y_13*y_2345,
   y_45*y_1234 - y_34*y_1245 +y_24*y_1345 -y_14*y_2345, 
    y_0*y_1234- y_34*y_12 +y_24*y_13 -y_23*y_14)

i67 : a_2: J23 == J17

o67 = true
\end{verbatim}

\subsubsection{Residual intersections in $E_7/P_7$}
\label{sec:typee7}

\SaveVerb{Q}=Q= 
\SaveVerb{f1}=f_1= 
\SaveVerb{f2}=f_2= 
\SaveVerb{f3}=f_3= 
\SaveVerb{f4}=f_4= 
\SaveVerb{f5}=f_5= 
\SaveVerb{f6}=f_6= 
\SaveVerb{f7}=f_7= 
\SaveVerb{I2}=I2= 
\SaveVerb{I3}=I3= 
\SaveVerb{I1}=I1= 

In this case the situation is very similar to that of $E_{6}/P_{6}$. We have two linked ideals (the \textit{Ulrich pair} in codimension 5) $ \UseVerb{I2} = I(Y^{2}) , \UseVerb{I3} = I(Y^3)$, and one ideal $ \UseVerb{I1} = I(Y^{1}) $ in codimension 6vhc. We include the relevant explicit calculations using Macaulay2, where we use the following notation:

\[
 \UseVerb{Q}= {p_{\emptyset}}^{*}, \UseVerb{f1} = {p_{7}}^{*}, \UseVerb{f2} = {p_{6}}^{*},
  \UseVerb{f3} = {p_{6}}^{*},  \UseVerb{f4} = {p_{4}}^{*}, \UseVerb{f5} =  {p_{2}}^{*}, \UseVerb{f6} =  {p_{3}}^{*}, \UseVerb{f7} =  {p_{1}}^{*},
  \UseVerb{y} = {p_{1}}^{*}.
\]

For more details we refer the reader to \cite{FTW20}.

\begin{figure}[h]
\label{graphe7p7}
\includegraphics[scale = 0.6]{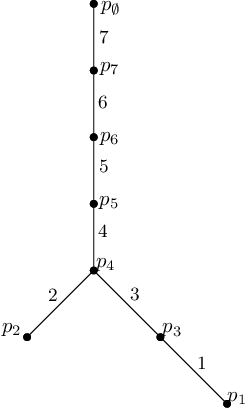}
\caption{The graph $\mathcal{G}_{7}$ for $E_{7}$}
\end{figure}

\begin{verbatim}
R = QQ[x_1,x_2,x_3,x_4,x_5,x_6,x_7,x_8,x_9,x_10,x_11,x_12,x_13,x_14, x_15,x_16,x_17,x_18,x_19,
x_20,x_21,x_22,x_23,x_24,x_25,x_26,x_27]

Qe7e6 = x_13*x_14*x_15-x_11*x_15*x_16-x_12*x_13*x_17+x_10*x_16*x_17+x_9*x_15*x_18-x_8*x_17*x_18+
x_11*x_12*x_19-x_10*x_14*x_19+x_5*x_18*x_19-x_7*x_15*x_20+x_6*x_17*x_20-x_4*x_19*x_20-x_9*x_12*x_21+
x_8*x_14*x_21-x_5*x_16*x_21+x_3*x_20*x_21+x_7*x_12*x_22-x_6*x_14*x_22+x_4*x_16*x_22-x_3*x_18*x_22+
x_9*x_10*x_23-x_8*x_11*x_23+x_5*x_13*x_23-x_2*x_20*x_23+x_1*x_22*x_23-x_7*x_10*x_24+x_6*x_11*x_24-
x_4*x_13*x_24+x_2*x_18*x_24-x_1*x_21*x_24+x_7*x_8*x_25-x_6*x_9*x_25+x_3*x_13*x_25-x_2*x_16*x_25+
x_1*x_19*x_25-x_5*x_7*x_26+x_4*x_9*x_26-x_3*x_11*x_26+x_2*x_14*x_26-x_1*x_17*x_26+x_5*x_6*x_27-
x_4*x_8*x_27+x_3*x_10*x_27-x_2*x_12*x_27+x_1*x_15*x_27

Q = Qe7e6
for i in 1..27 list f_i = diff(x_i, Q)
I1 = ideal(Q, f_1, f_2,f_3, f_4, f_5, f_6, f_8, f_10, f_12, f_15, x_27)
I3 = ideal(Q, f_1, f_2,f_3, f_4, f_5)
I51 = ideal(Q,f_1,f_2,f_3,f_4,f_6,f_7)
J = ideal(Q,f_1,f_2,f_3,f_4,f_6)
i16 : J:I2 == I1 
o16 = true
I = ideal(Q,f_1,f_2,f_3,f_4)
i30 : I:I2 == I3
o30 = true
 

\end{verbatim}

\section{Non-minuscule examples}

For type $E_{n}$, we will be interested in $X^{w}$ for $w = s_{i} \cdots s_{4} s_{2}$ given by consecutive (non-repeating) reflections on the Dynkin diagram, with $i = 1,3,5,6,..., n$. We will abbreviate $X^{w}$ as just $X^{i}$ as before. We will denote the corresponding opposite Demazure module by $M$. 

\subsection{$E_{6}$}

\subsubsection{Codimension three}
For $i=3$  we have the following decomposition

\begin{align*}
\operatorname{res}^{\mathfrak{g}}_{\mathfrak{g}^{i}} ( V(\omega_{2}) ) ={ \color{blue}V(\omega_{6})} \oplus {\color{red} V(\omega_{1}+\omega_{4}) } & \oplus V(\omega_{2}+ \omega_{6})  \oplus V(\omega_{1}+\omega_{5}) \oplus V(\omega_{2}) \\
 & \oplus V(2 \omega_{2}) 
\end{align*}

\noindent  The Pl\"ucker coordinates vanishing on $X^{3}$ are those dual to ${ \color{blue}V(\omega_{6})}$. These are just the extremal Pl\"ucker coordinates $p_{w}$ where $w$ does not involve the reflection $s_{3}$. Now, $w \cdot v$ is the lowest weight vector for ${\color{red} V(\omega_{1}+\omega_{4})}$. Since the entirety of this $i$- graded component is. in $M$, all higher components are in $M$ as well. 

\subsubsection{Codimension four}

For $i =1$: 

\begin{align*}
\operatorname{res}^{\mathfrak{g}}_{\mathfrak{g}^{i}} ( V(\omega_{2}) )={ \color{blue}V(\omega_{3})} &\oplus V(\omega_{5})  \oplus V(\omega_{2}) \\
&\oplus {\color{blue} \mathbb{C}}
\end{align*}
 
 \noindent The extremal Pl\"ucker coordinates are dual to ${ \color{blue}V(\omega_{3})}$. However, an additional, non extremal, Pl\"ucker coordinate, also vanishes on $X^{1}$.

 \subsection{$E_{7}$}
 
 \subsubsection{Codimension three} (See \cite{KHLJW18} for this subsection.)
 
 For $i =3$:
 
 \begin{align*}
 \operatorname{res}^{\mathfrak{g}}_{\mathfrak{g}^{i}} ( V(\omega_{2}) ) = {\color{blue} V(\omega_{7})} \oplus V(\omega_{1} + \omega_{5}) \oplus \cdots \oplus V(\omega_{1}+ \omega_{5}) \oplus V(\omega_{2})
 \end{align*}
 
 \noindent The six extremal Pl\"ucker coordinates $p_{w}$ where $w$ doesn't involve $s_{3}$ span $V(\omega_{2}) \subset V$ dual to $V(\omega_{7})$ and cut out $X^{3}$ set-theoretically. 
 
  For $i =5$:
 
 \begin{align*}
 \operatorname{res}^{\mathfrak{g}}_{\mathfrak{g}^{i}} ( V(\omega_{2}) ) = {\color{blue} V(\omega_{1})} \oplus V(\omega_{4} + \omega_{7}) \oplus \cdots \oplus V(\omega_{3}+ \omega_{6}) \oplus V(\omega_{2})
 \end{align*}
 
 \noindent The five extremal Pl\"ucker coordinates $p_{w}$ cut out $X^{5}$ set-theoretically.

 \subsubsection{Codimension four}
 
 For $i = 1$: 
 
 \begin{align*}
  \operatorname{res}^{\mathfrak{g}}_{\mathfrak{g}^{i}} ( V(\omega_{2}) ) = {\color{blue} V(\omega_{2})} &\oplus {\color{red}V(\omega_{5})} \oplus {\color{orange} V(\omega_{2})  \oplus V(\omega_{5}) \oplus V(\omega_{2})} \\
  &\oplus {\color{blue} V(\omega_{7})} \oplus { \color{orange} V(\omega_{3}+ \omega_{7})  \oplus V(\omega_{7})}
  \end{align*}
 
 There are thirty-two extremal Pl\"ucker coordinates dual to $V(\omega_{2})$ which generate $X^{1}$ and twelve more non-extremal ones dual to $V(\omega_{7})$.

 For $i = 6$:
 
 \begin{align*}
   \operatorname{res}^{\mathfrak{g}}_{\mathfrak{g}^{i}} ( V(\omega_{2}) ) = {\color{blue} V(\omega_{5})} & \oplus {\color{red}V(\omega_{3} + \omega_{7})} & \oplus  V(\omega_{1}+ \omega_{5}) & \oplus V(\omega_{1} + \omega_{7}) & \oplus V(\omega_{1} + \omega_{2}) & \oplus V(\omega_{3} + \omega_{1}) \oplus V(\omega_{2})\\
& \oplus {\color{blue}V(\omega_{7})} & \oplus  V(\omega_{2}) & \oplus V(\omega_{4} + \omega_{7}) & \oplus V(\omega_{5}) & \oplus V(\omega_{7}) \\
&  & \oplus  V(\omega_{2} + 2 \omega_{7}) & \oplus V(\omega_{1} + \omega_{7}) & \oplus V(\omega_{5} + 2 \omega_{1}) & 
 \end{align*}
 
 \noindent There are sixteen extremal Pl\"ucker coordinates dual to ${ \color{blue}V(\omega_{5}) }$ vanishing on $X^{6}$ and two more non-extremal Pl\"ucker coordinates dual to ${\color{blue} V(\omega_{7})}$.

 \subsubsection{Codimension five}
 
 For $i = 7$: 
 
 \begin{align*}
    \operatorname{res}^{\mathfrak{g}}_{\mathfrak{g}^{i}} ( V(\omega_{2}) ) = {\color{blue} V(\omega_{2})}  & \oplus {\color{red}V(\omega_{5})}  \oplus V(\omega_{3}) \oplus V(\omega_{2}) \\
     & \oplus { \color{blue} V(\omega_{1}) }  \oplus V(\omega_{6})
 \end{align*}
 
  \noindent Dual to ${\color{blue} V(\omega_{2})}$, there are seventy-two extremal Pl\"ucker coordinates as well as six non-extremal Pl\"ucker coordinates vanishing on $X^{7}$. There are twenty-seven more non-extremal Pl\"ucker coordinates dual to ${\color{blue} V(\omega_{1})}$.

\begin{rmk}

In forthcoming work, we plan to investigate and describe the intersection of the opposite Schubert varieties with the opposite big open cell, $Y^{j}$, as was carried out for the minuscule cases in previous sections.
\end{rmk}

\bibliographystyle{alpha}


\end{document}